\pdfoutput=1
\documentclass[leqno]{article}
\usepackage{calc,microtype}
\usepackage[a4paper, margin = 25mm]{geometry}
\RequirePackage[raiselinks=false,colorlinks=true,citecolor=blue,urlcolor=blue,linkcolor=blue,bookmarksopen=true]{hyperref}
\newcommand{\arxiv}[1]{\href{http://arxiv.org/pdf/#1}{arXiv:#1}}

\makeatletter
\def\@cite#1#2{[{#1\if@tempswa ,~#2\fi}]}% NEW
\makeatother

\usepackage[arrow, matrix, tips, curve, graph, rotate]{xy}
\SelectTips{cm}{10}
\newcommand{\pullbackcorner}[1][dr]{\save*!/#1-1.5pc/#1:(-1,1)@^{|-}\restore}
\newcommand{\cd}[2][]{\vcenter{\hbox{\xymatrix#1{#2}}}}
\newdir{ >}{{}*!/-5pt/\dir{>}}

\DeclareMathAlphabet{\mathbf}{OT1}{cmr}{b}{n}

\usepackage{amsmath,amsthm,amssymb}
\usepackage{eucal}

\DeclareMathAlphabet{\mathbbe}{U}{bbold}{m}{n}

\newcommand{\C}{\mathcal{C}}

\newcommand{\simplexcategory}{\mathbbe{\Delta}}
\newcommand{\Decbot}[1]{\operatorname{Dec}_\bot{}\kern-2pt{#1}}
\newcommand{\Dectop}[1]{\operatorname{Dec}_\top{}\kern-2pt{#1}}

\newcommand{\III}{\Lambda}

\newcommand{\op}{^{\text{{\rm{op}}}}}

\newtheorem{lem}[subsection]{Lemma}
\newtheorem{propo}[subsection]{Proposition}
\newtheorem*{thm}{Theorem}
\theoremstyle{definition}
\newtheorem*{defi}{Definition}
\theoremstyle{remark}

\makeatletter
\providecommand*\xRightarrow[2][]{\ext@arrow 04{22}0{\Rightarrowfill@}{#1}{#2}}
\makeatletter

\newcommand{\authorwithaffiliation}[3]{\parbox[t]{#3}{\begin{center}{\sc #1}\\[2pt] \footnotesize #2\end{center}}}

\begin{document}

\title{EVERY $2$-SEGAL SPACE IS UNITAL}
\author{\authorwithaffiliation{Matthew Feller}{University of Virginia}{9em}
\qquad  
\authorwithaffiliation{Richard Garner}{Macquarie University}{9em}
\qquad 
\authorwithaffiliation{Joachim Kock}{Univ. Aut\`onoma de 
Barcelona}{10em}
\\[-2ex]
\authorwithaffiliation{May U. Proulx}{University of Leicester}{8em}
\qquad
\authorwithaffiliation{Mark Weber}{Macquarie University}{8em}
}

\date{}

\maketitle

%%%%%%%%%%%%%%%%%%%%%%%%%%%%%%%%%%%%%%%%%%%%%%%%%%%%%%%%%%%%%%%%%%%%%%%%%%%%%%%%
\section*{Introduction}
%%%%%%%%%%%%%%%%%%%%%%%%%%%%%%%%%%%%%%%%%%%%%%%%%%%%%%%%%%%%%%%%%%%%%%%%%%%%%%%%

$2$-Segal spaces were introduced by Dyckerhoff and
Kapranov~\cite{Dyckerhoff-Kapranov:1212.3563} for applications in
representation theory, homological algebra, and geometry, motivated in
particular by Waldhausen's $S$-construction and Hall algebras. A $2$-Segal
space is a simplicial space $X$ such that for every triangulation $T$ of
every convex plane $n$-gon (for $n \geqslant 2$), we have $X_n \simeq
\lim_{t\in T} X(t)$. Independently, a little later, G\'alvez-Carrillo,
Kock, and Tonks~\cite{Galvez-Kock-Tonks:1512.07573} introduced the notion
of \emph{decomposition space} for applications in combinatorics, in
connection with M\"obius inversion. A decomposition space is a simplicial
space $X \colon \simplexcategory\op \rightarrow \mathcal{S}$ for
which all pushouts of active maps along inert maps in $\simplexcategory$
are sent to pullbacks in $\mathcal{S}$. Here, the \emph{inert} maps in
$\simplexcategory$ are generated by the outer coface maps, while the
\emph{active} maps are generated by the codegeneracy and inner coface maps.
The condition satisfied by $X$ with respect to pushouts of outer coface
maps against inner ones is precisely equivalent to the $2$-Segal condition.
For Dyckerhoff and Kapranov, the condition for pushouts of outer cofaces
against codegeneracies is a further axiom which they call {\em unitality}
\cite[Definition 2.5.2]{Dyckerhoff-Kapranov:1212.3563}. Thus, decomposition
spaces are the same thing as unital $2$-Segal spaces. While the $2$-Segal
axiom is expressly the condition required in order to induce a
(co)associative (co)multiplication on the linear span of $X_1$, the
unitality condition ensures that this (co)multiplication is (co)unital,
which is an important property in many applications.

The present note shows that the unitality condition is actually 
automatic, by proving:

\medskip

\noindent \textbf{Theorem.} \emph{Every $2$-Segal space is unital.}

\medskip

This result is unexpected, as it is
% Firstly, it cannot be derived by the standard tricks
% with pullback squares; secondly, it is
not so common in mathematics for (co)associativity to imply (co)unitality.

%%%%%%%%%%%%%%%%%%%%%%%%%%%%%%%%%%%%%%%%%%%%%%%%%%%%%%%%%%%%%%%%%%%%%%%%%%%%%%%%
\section{Definitions and theorem}
%%%%%%%%%%%%%%%%%%%%%%%%%%%%%%%%%%%%%%%%%%%%%%%%%%%%%%%%%%%%%%%%%%%%%%%%%%%%%%%%

In order to cover all flavours of $2$-Segal space that appear in the
literature, we give a proof which applies both to $2$-Segal objects in an
$\infty$-category with finite limits and to $2$-Segal objects in a Quillen
model category.
From now on, $\C$ will denote either an $\infty$-category with finite
limits or a Quillen model category. In the latter case, ``pullback''
will mean a (strictly commuting) homotopy pullback.

\begin{defi}
  (cf.~\cite{Dyckerhoff-Kapranov:1212.3563},~\cite{Galvez-Kock-Tonks:1512.07573})
  A simplicial object $X\colon \simplexcategory\op\to\C$ is called
  \emph{$2$-Segal} when the commuting squares that express the simplicial
  identities between inner and outer face maps of $X$ are
  pullback squares.
    More precisely, for all  $0<i<n$ we have
  pullbacks
  \begin{equation*}
    \cd[@!@-1.2em]{
      X_{n+1} \pullbackcorner \ar[r]^-{d_{i+1}} \ar[d]_-{d_0} & X_n \ar[d]^-{d_0} \\
      X_n \ar[r]_-{d_{i}} & X_{n-1}
    } \qquad \text{and} \qquad
    \cd[@!@-1.2em]{
      X_{n+1} \pullbackcorner \ar[r]^-{d_{i}} \ar[d]_-{d_{n+1}} & X_n \ar[d]^-{d_n} \\
      X_n \ar[r]_-{d_{i}} & X_{n-1}\rlap{ .}
    }
  \end{equation*}
  We say that $X$ is \emph{upper $2$-Segal} when only squares as to the
  left are required to be pullbacks, and {\em lower $2$-Segal} when this is
  only required for squares as to the right.
  \footnote{For our purposes,
  splitting into upper $2$-Segal and lower $2$-Segal is just for economy;
  in the theory of higher Segal spaces~\cite{Poguntke:1709.06510}
  ($k$-Segal spaces for $k>2$), the distinction between upper and lower
  becomes essential.}
\end{defi}

\begin{defi}
  A $2$-Segal space $X$ is called \emph{unital} if for all $0\leqslant i \leqslant n$
  the following
  squares are pullbacks:
  \begin{equation*}
    \cd[@!@-1em]{
      X_{n+1} \pullbackcorner \ar[r]^-{s_{i+1}} \ar[d]_-{d_0} & X_{n+2} \ar[d]^-{d_0} \\
      X_{n} \ar[r]_-{s_i} & X_{n+1}
    } \qquad 
	\qquad
    \cd[@!@-1em]{
      X_{n+1} \pullbackcorner \ar[r]^-{s_i} \ar[d]_-{d_{n+1}} & X_{n+2} 
	  \ar[d]^-{d_{n+2}} \\
      X_n \ar[r]_-{s_i} & X_{n+1}\rlap{ .}
    }
  \end{equation*}
  We call an upper $2$-Segal space \emph{upper unital} when only the
  pullbacks on the left are required, and call a lower $2$-Segal space
  \emph{lower unital} when only the pullbacks on the right are required.
\end{defi}

% \begin{defi}
%   A $2$-Segal space $X$ is called \emph{unital} if the following two
%   squares are pullbacks:
%   \begin{equation}\label{eq:1}
%     \cd[@!]{
%       X_1 \pullbackcorner \ar[r]^-{s_1} \ar[d]_-{d_0} & X_2 \ar[d]^-{d_0} \\
%       X_0 \ar[r]_-{s_0} & X_1
%     } \qquad 
% 	\qquad
%     \cd[@!]{
%       X_1 \pullbackcorner \ar[r]^-{s_0} \ar[d]_-{d_1} & X_2 \ar[d]^-{d_2} \\
%       X_0 \ar[r]_-{s_0} & X_1\rlap{ .}
%     }
%   \end{equation}
%   We call an upper $2$-Segal space \emph{upper unital} when only the
%   pullback on the left is required, and call a lower $2$-Segal space
%   \emph{lower unital} when only the pullback on the right is required.
% \end{defi}
% 
% Note that in \cite{Dyckerhoff-Kapranov:1212.3563}, unitality is
% defined by the apparently stronger requirement that \emph{every}
% square which commutes a degeneracy past an outer face map should be a
% pullback. This discrepancy is accounted for by
% \cite[Proposition~3.5]{Galvez-Kock-Tonks:1512.07573}, which shows that
% in a $2$-Segal space, all degeneracy--outer face squares except for
% those in~\eqref{eq:1} are automatically pullbacks. The argument uses
% the $2$-Segal axioms, the simplicial identities, and the
% cancellativity of pullbacks (Lemma~\ref{prismlem} below).
% 
% By extending the argument
% of~\cite[Proposition~3.5]{Galvez-Kock-Tonks:1512.07573}, we will show
% that in a $2$-Segal space, the squares in~\eqref{eq:1} are
% \emph{also} forced to be pullbacks. In other words, we will prove:

\begin{thm}
  Every $2$-Segal space is unital.  More precisely, every upper 
  $2$-Segal space is upper unital, and every lower $2$-Segal space is 
  lower unital.
\end{thm}

%%%%%%%%%%%%%%%%%%%%%%%%%%%%%%%%%%%%%%%%%%%%%%%%%%%%%%%%%%%%%%%%%%%%%%%%%%%%%%%%
\section{The proof}
%%%%%%%%%%%%%%%%%%%%%%%%%%%%%%%%%%%%%%%%%%%%%%%%%%%%%%%%%%%%%%%%%%%%%%%%%%%%%%%%

By symmetry, it is enough to prove:
\begin{propo}\label{upper}
  If $X \colon\simplexcategory\op\to\C$ is upper $2$-Segal,
  then it is also upper unital.
\end{propo}

We do so using two lemmas, which are standard both in
$\infty$-category theory and model category theory. 

\begin{lem}[Prism Lemma]\label{prismlem}
  Given a commuting diagram
  \begin{equation*}
	\cd[@!]{
	  {A} \ar[r]^-{} \ar[d]_{} &
	  {B} \ar[r]^-{} \ar[d]_{} &
	  {C} \ar[d]^{} \\
	  {D} \ar[r]_-{} &
	  {E} \ar[r]_-{} &
	  {F}
	}
  \end{equation*}
  (formally a $\Delta^1 \times \Delta^2$-diagram in the
  $\infty$-category case), suppose the right-hand square is a 
  pullback. Then
  the outer rectangle is a pullback if and only if the left-hand square
  is a pullback.
\end{lem}

\begin{proof}
  For the $\infty$-category version, see (the dual
  of)~\cite[Lemma~4.4.2.1]{Lurie:HTT}. The model category version is
  proven in the right-proper case
  in~\cite[Proposition~13.3.9]{Hirschhorn}; we give the
  general case in the appendix.
\end{proof}

\begin{lem}\label{retractlem}
  Pullback squares are stable under retract.
\end{lem}

\begin{proof}
  The $\infty$-category version follows from (the dual of)
  \cite[Lemma~5.1.6.3]{Lurie:HTT}. The model category version is known 
  to experts, but since we do not know of any reference, we give a proof in 
  the appendix.
\end{proof}

\begin{proof}[Proof of Proposition~\ref{upper}]
  We first establish the pullback condition for $n \geqslant
   1$ and $0 \leqslant i \leqslant n$ by following the argument
   of~\cite[Proposition~3.5]{Galvez-Kock-Tonks:1512.07573},
   exploiting that every degeneracy map except $s_0\colon
   X_0\rightarrow X_1$ is a section of an inner face map. Explicitly, if we choose $j\in\{i,i+1\}$ with $0 < j \leqslant
   n$, then $s_{i} \colon X_{n} \rightarrow X_{n+1}$ is a
   section of the inner face map $d_{j} \colon X_{n+1}
   \rightarrow X_{n}$ and $s_{i+1} \colon X_{n+1}
   \rightarrow X_{n+2}$ is a section of $d_{j+1}$, forming
   the prism diagram to the left below. Here the outer
  square is a pullback since its top and bottom edges are the images
  of identity maps in $\simplexcategory$, while the right-hand square
  is a pullback since $X$ is upper $2$-Segal and $d_j$ is an 
  inner face map. So by Lemma~\ref{prismlem}, the left-hand square is a pullback as required.
  \begin{equation*}
    \cd[@!@-0.8em]{
      {X_{n+1}} \ar[r]^-{s_{i+1}} \ar[d]_{d_0} &
      {X_{n+2}} \ar[d]^{d_0} \ar[r]^-{d_{j+1}} &
      X_{n+1} \ar[d]^-{d_0} \\
      {X_n} \ar[r]_-{s_i} &
      X_{n+1} \ar[r]_-{d_{j}} &
      {X_n}
    } \qquad
    \cd[@!@-1.7em]{
      {X_1} \ar@{->}[rrrr]^-{s_1} \ar[dd]_{d_0} \ar[dr]^-{s_1} & & & & 
      {X_2} \ar@{->}[rrrr]^-{d_1} \ar[dd]_(0.3){d_0}|-\hole
      \ar[dr]^-{s_2} & & &
      & 
      {X_1} \ar[dd]_(0.3){d_0}|-\hole \ar[dr]^-{s_1} \\ &
      {X_2} \ar@{->}[rrrr]^-{s_1} \ar[dd]^(0.3){d_0}  & & & & 
      {X_3} \ar@{->}[rrrr]^-{d_1} \ar[dd]^(0.3){d_0}  & & & & 
      {X_2} \ar@{->}[dd]^{d_0} \\ 
      {X_0} \ar@{->}[rrrr]_-{s_0}|(0.25)\hole \ar[dr]_-{s_0} & & & & 
      {X_1} \ar@{->}[rrrr]_-{d_0}|(0.25)\hole \ar[dr]_-{s_1}& & & & 
      {X_0} \ar[dr]_-{s_0}\\ &
      {X_1} \ar@{->}[rrrr]_-{s_0} & & & & 
      {X_2} \ar@{->}[rrrr]_-{d_0} & & & & 
      {X_1}
    }
  \end{equation*}
  The remaining case, which is not covered
  by~\cite[Proposition~3.5]{Galvez-Kock-Tonks:1512.07573}, is the
  square with $n=i=0$. To see that this is a pullback, we
  exhibit it as a retract of the square for $n=i=1$, as displayed above
  right. Since we already know the $n=i=1$ square is a pullback, so is
  the $n=i=0$ square by Lemma~\ref{retractlem}.
  % For this, we consider the diagram to the right. Here the last (upper) unitality square is exhibited as a retract of the
  % square for $n=i=1$, which is already known to be a pullback. We
  % conclude by Lemma~\ref{retractlem}.
%   the right-hand square is a pullback since $X$ is upper $2$-Segal, while
%   the outer square is a pullback since its top and bottom edges are the
%   images of identity maps in $\simplexcategory$. Therefore, by
%   Lemma~\ref{prismlem}, the left-hand square $\star$ (case $n=i=1$) is also a pullback.
% %   The same argument shows that all the left-unitality squares for $n\geq 1$
% %   are pullbacks (see \cite[Lemma 3.5]{Galvez-Kock-Tonks:1512.07573}). Indeed, for $n\geq 1$, every degeneracy map $s_i: X_n \to 
% %   X_{n+1}$ is a section of an inner 
% %   face map, and therefore the corresponding upper-unitality square is a 
% %   section of a upper-$2$-Segal square, as in the $1=i=n$ case shown.
% %   The only remaining 
% %   upper-unitality square is the $0=i=n$ case, which sits as 
% %   the left and right faces of
% %   the diagram above right, and is thereby exhibited as a retract of the middle face 
% %   $\star$.
%   On the other hand, the diagram above right is exhibited as the upper-unitality square---the left and right faces of the diagram---as a retract of the middle face $\star$. Thus, by Lemma~\ref{retractlem}, this remaining square (case $n=i=0$) is a pullback too.
%   Note that case $n=i=1$ only used Lemma~\ref{prismlem}. By setting up the appropriate prisms, the other cases for $n\geq 1$ follow similarly. See  for details.
\end{proof}

%%%%%%%%%%%%%%%%%%%%%%%%%%%%%%%%%%%%%%%%%%%%%%%%%%
\section*{Appendix}
%%%%%%%%%%%%%%%%%%%%%%%%%%%%%%%%%%%%%%%%%%%%%%%%%%

\looseness=-1 We provide proofs of the two lemmas in the context of a model category
$\C$. First we recall the notion of (strictly commuting) homotopy
pullback. Writing $\III$ for the cospan category
$0 \rightarrow 2 \leftarrow 1$, we endow $\C^{\III}$ with the
\emph{injective} model structure, whose weak equivalences and
cofibrations are pointwise, and whose fibrant objects are cospans of
fibrations between fibrant objects in $\C$. A commuting square in
$\C$, as to the left in
\begin{equation*}
  \cd[@-0.2em]{
    {P} \ar[r]^-{} \ar[d]_{} &
    {A_1} \ar[d]^{} & &
    A_0 \ar[d]^*-<0.7em>[@!-90]{\scriptstyle\sim} \ar[r]^-{} & A_2 \ar[d]^*-<0.7em>[@!-90]{\scriptstyle\sim} & A_1
    \ar[l]_-{} \ar[d]^*-<0.7em>[@!-90]{\scriptstyle\sim}\\
    {A_0} \ar[r]_-{} &
    {A_2} & &
    A'_0 \ar@{->>}[r] & A'_2 & A'_1 \ar@{->>}[l]
  }
\end{equation*}
is a \emph{homotopy pullback} if for some (equivalently, any)
fibrant replacement in
$\smash{\C^{\III}}$ for its
underlying cospan, as displayed to the right above, the induced map $\smash{P \rightarrow A_0'
\times_{A'_2} A'_1}$ into the strict pullback is a weak equivalence.

\begin{proof}[Proof of Lemma~\ref{prismlem} in the model category case.]
  We first replace $D \rightarrow E \rightarrow F \leftarrow C$ by a
  diagram of fibrations between fibrant objects, as to the left in:
  \begin{equation*}
    \cd[@!@-1.8em]{
      {A} \ar[rr]^-{} \ar[dd]_{}  & &
      {B} \ar[rr]^-{} \ar[dd]_{} & &
      {C} \ar[dd]^{} \ar[dr]^(.4)*-<0.7em>[@!-45]{\scriptstyle\sim} \\ & & & & &
      {C'} \ar@{->>}[dd]^{} \\ 
      {D} \ar[rr]_-{} \ar[dr]^(.4)*-<0.7em>[@!-45]{\scriptstyle\sim}& &
      {E} \ar[rr]_-{} \ar[dr]^(.4)*-<0.7em>[@!-45]{\scriptstyle\sim}& &
      {F} \ar[dr]^(.4)*-<0.7em>[@!-45]{\scriptstyle\sim}\\ &
      {D'} \ar@{->>}[rr]_-{} & &
      {E'} \ar@{->>}[rr]_-{} & &
      {F'}
    } \qquad \qquad 
        \cd[@!@-1.8em]{
      {A} \ar[rr]^-{} \ar[dd]_{} \ar[dr]^-{} & &
      {B} \ar[rr]^-{} \ar[dd]_{}|-\hole 
	  \ar[dr] & &
      {C} \ar[dd]^{}|-\hole \ar[dr]^(.4)*-<0.7em>[@!-45]{\scriptstyle\sim} \\ &
      {A'} \ar[rr]^-{} \ar@{->>}[dd]_{} \pullbackcorner & &
      {B'} \ar[rr]^-{} \ar@{->>}[dd]_{} \pullbackcorner & &
      {C'} \ar@{->>}[dd]^{} \\ 
      {D} \ar[rr]_-{}|-\hole \ar[dr]^(.4)*-<0.7em>[@!-45]{\scriptstyle\sim}& &
      {E} \ar[rr]_-{}|-\hole \ar[dr]^(.4)*-<0.7em>[@!-45]{\scriptstyle\sim}& &
      {F} \ar[dr]^(.4)*-<0.7em>[@!-45]{\scriptstyle\sim}\\ &
      {D'} \ar@{->>}[rr]_-{} & &
      {E'} \ar@{->>}[rr]_-{} & &
      {F'}\rlap{ .}
    }
  \end{equation*}
  By taking strict pullbacks we complete this to the diagram as to the
  right. Since the right-hand back face is assumed to be a homotopy
  pullback, $B \rightarrow B'$ is a weak equivalence, and so
  $D' \twoheadrightarrow E' \twoheadleftarrow B'$ is a fibrant
  replacement for $D \rightarrow E \leftarrow B$ in $\C^{\III}$. Thus
  $A \rightarrow A'$ is a weak equivalence exactly when the
  left-hand back face is a homotopy pullback. Since
  $D' \twoheadrightarrow F' \twoheadleftarrow C'$ is a fibrant
  replacement for $D \rightarrow F \leftarrow C$, we also have that
  $A \rightarrow A'$ is a weak equivalence exactly when the back
  rectangle is a homotopy pullback, as desired.
\end{proof}

\begin{proof}[Proof of Lemma~\ref{retractlem} in the model category case.]
  Suppose given a homotopy pullback square in $\C$, together with a retract
  of it in the category of commutative squares in $\C$, as to the left in:
  \begin{equation*}
    \cd[@!@-1.9em]{
      {Q} \ar@{->}[rr]^-{} \ar[dd]_{} \ar[dr]^-{} & &
      {P} \ar@{->}[rr]^-{} \ar[dd]_{}|-\hole \ar[dr]^-{} & &
      {Q} \ar[dd]^{}|-\hole \ar[dr]^-{} \\ &
      {B_1} \ar@{->}[rr]^-{} \ar[dd]_{}  & &
      {A_1} \ar@{->}[rr]^-{} \ar[dd]_{}  & &
      {B_1} \ar@{->}[dd]^{} \\ 
      {B_0} \ar@{->}[rr]_-{}|-\hole \ar[dr]^-{}& &
      {A_0} \ar@{->}[rr]_-{}|-\hole \ar[dr]^-{}& &
      {B_0} \ar[dr]^-{}\\ &
      {B_2} \ar@{->}[rr]_-{} & &
      {A_2} \ar@{->}[rr]_-{} & &
      {B_2}
    } \qquad \qquad 
    \cd[@!]{
      {\Delta Q} \ar[r]^-{} \ar[d]_{} &
      {\Delta P} \ar[r]^-{} \ar[d]_{} &
      {\Delta Q} \ar[d]^{} \\
      {B} \ar[r]_-{} &
      {A} \ar[r]_-{} & B\rlap{ .}
    }
  \end{equation*}
  We must show that the left 
%   (=right)
  (equally, the right) 
  face of this diagram
  is also a homotopy pullback. Regarding
  $B_0 \rightarrow B_2 \leftarrow B_1$ and
  $A_0 \rightarrow A_2 \leftarrow A_1$ as objects of $\C^{\III}$, and
  regarding
  $Q$ and $P$ as constant objects $\Delta Q$ and $\Delta P$, we obtain
  a retract diagram of arrows in $\C^{\III}$ as displayed to the right
  above. We will now fibrantly replace $A$ and $B$ in $\C^{\III}$ in
  such a way as to obtain a new retract diagram $B' \to A' \to B'$. To
  this end, we first fibrantly replace $B$ via a trivial cofibration
  $B
  \mathrel{\stackrel{\scriptscriptstyle\sim}{\vphantom{a}\smash{\rightarrowtail}}}
  B'$. Now we factor the composite $A \rightarrow B \rightarrow B'$ as
  a trivial cofibration
  $A
  \mathrel{\smash{\stackrel{\scriptscriptstyle\sim}{\vphantom{a}\smash{\rightarrowtail}}}}
  A'$ followed by a fibration $A' \twoheadrightarrow B'$. Finally, we
  take a lifting in the square
  \begin{equation*}
    \cd{
      {B} \ar[r]^-{} \ar@{ >->}[d]^*-<0.7em>[@!-90]{\scriptstyle\sim}& A \ar[r]^-{} & 
      {A'} \ar@{->>}[d]^{} \\
      {B'} \ar@{=}[rr] \ar@{-->}[urr]_-{}& &
      {B'}\rlap{ .}
    }
  \end{equation*}
  Altogether, we now have a retract diagram in the category of
  composable pairs in $\C^{\III}$, as to the left in:
\begin{equation*}
    \cd[@!@-2.3em]{
      {\Delta Q} \ar[rr]^-{} \ar[dd]_{}  & &
      {\Delta P} \ar[rr]^-{} \ar[dd]_{} & &
      {\Delta Q} \ar[dd]^{}  \\ \\
      {B} \ar[rr]_-{} \ar[dr]^(.4)*-<0.7em>[@!-45]{\scriptstyle\sim}& &
      {A} \ar[rr]_-{} \ar[dr]^(.4)*-<0.7em>[@!-45]{\scriptstyle\sim}& &
      {B} \ar[dr]^(.4)*-<0.7em>[@!-45]{\scriptstyle\sim}\\ &
      {B'} \ar[rr]_-{} & &
      {A'} \ar[rr]_-{} & &
      {B'}
    } \qquad \qquad 
        \cd[@!@-2.6em]{
      {\Delta Q} \ar[rr]^-{} \ar[dd]_{} \ar[dr]^-{} & &
      {\Delta P} \ar[rr]^-{} \ar[dd]_{}|-\hole \ar[dr]^-{} & &
      {\Delta Q} \ar[dd]^{}|-\hole \ar[dr]^-{} \\ &
      {\Delta Q'} \ar[rr]^-{} \ar[dd]_{} & &
      {\Delta P'} \ar[rr]^-{} \ar[dd]_{} & &
      {\Delta Q'} \ar[dd]^{} \\ 
      {B} \ar[rr]_-{}|-\hole \ar[dr]^-{}& &
      {A} \ar[rr]_-{}|-\hole \ar[dr]^-{}& &
      {B} \ar[dr]^-{}\\ &
      {B'} \ar[rr]_-{} & &
      {A'} \ar[rr]_-{} & &
      {B'}\rlap{ .}
    }
  \end{equation*}
  By forming the strict pullbacks $Q'$ and $P'$ of the cospans $B'$ and
  $A'$ we may complete this to the retract diagram as to the right above; note
  in particular that the map $Q \rightarrow Q'$ in $\C$ is a retract of the
  map $P \rightarrow P'$. Since $\Delta P \rightarrow A$ describes
  a homotopy pullback, $P \rightarrow P'$ is a weak equivalence; so
  its retract $Q \rightarrow Q'$ is also a weak equivalence, which is to 
  say that $\Delta Q \rightarrow B $ also describes a homotopy pullback.
\end{proof}

\medskip

  \footnotesize
  
  \noindent {\bf Acknowledgments.}
  R.G. was supported by Australian Research Council grants DP160101519 and 
  FT160100393. J.K. was supported by grants
  MTM2016-80439-P (AEI/FEDER, UE) of Spain and 2017-SGR-1725 of
  Catalonia. M.W. was supported by Czech Science Foundation grant GA CR P201/12/G028.

\medskip

\noindent E-mail addresses:\\
Matthew Feller \url{<feller@virginia.edu>}\\
Richard Garner \url{<richard.garner@mq.edu.au>}\\
Joachim Kock \url{<kock@mat.uab.cat>}\\
May U. Proulx \url{<meup1@leicester.ac.uk>}\\
Mark Weber \url{<mark.weber.math@googlemail.com>}
\end{document}